\documentclass[12pt]{amsart}
\usepackage{amssymb}
\usepackage{epsfig}
\usepackage{graphicx}

\theoremstyle{plain}
\newtheorem{prop}{Proposition}[section]
\newtheorem{theo}[prop]{Theorem}

\newtheorem{lemm}[prop]{Lemma}

\theoremstyle{definition}

\newtheorem{rema}[prop]{Remark}

\newtheorem{exam}[prop]{Example}

\newcommand{\rank}{\mathrm{rank}}

\newcommand{\Sym}{\mathrm{Sym }}

\newcommand{\rH}{\mathrm{H}}

\newcommand{\tors}{{\text{tors}}}
\newcommand{\dt}{\mathord{\cdot}}
\newcommand{\divides}{\mathbin{|}}
\newcommand{\ndivides}{\nmid}
\newcommand{\injection}{\hookrightarrow}
\newcommand{\eps}{\varepsilon}

\newcommand{\ra}{\rightarrow}

\newcommand{\bF}{{\mathbb F}}

\newcommand{\bN}{{\mathbb N}}
\newcommand{\bP}{{\mathbb P}}
\newcommand{\bQ}{{\mathbb Q}}
\newcommand{\bR}{{\mathbb R}}

\newcommand{\bZ}{{\mathbb Z}}

\newcommand{\cN}{{\mathcal N}}
\newcommand{\cO}{{\mathcal O}}

\newcommand{\cT}{{\mathcal T}}

\newcommand{\fS}{{\mathfrak S}}

\author{David Harvey, Brendan Hassett, and Yuri Tschinkel}
\title[Characterizing projective spaces]{Characterizing projective spaces on deformations of Hilbert schemes of
K3 surfaces}

\begin{document}
\date{\today}

\maketitle

\section{Introduction} \label{sect:intro}

Let $X$ be an irreducible holomorphic symplectic manifold, i.e.,
a compact K\"ahler simply-connected manifold admitting a unique
nondegenerate holomorphic two-form.  
Let $\left(,\right)$ denote the Beauville--Bogomolov form on the 
cohomology group $\rH^2(X,\bZ)$, normalized so that it is integral
and primitive. 
When $X$ is a K3 surface
this coincides with the intersection form.  
In higher dimensions, the form induces an inclusion
\begin{equation} \label{eqn:incl}
\rH^2(X,\bZ) \subset \rH_2(X,\bZ),
\end{equation}
which allows us to extend $\left(,\right)$ to a $\bQ$-valued quadratic form.  

Lagrangian projective spaces play a fundamental r\^ole in the birational
geometry of these classes of manifolds. 
If $X$ contains a holomorphically embedded projective space $\bP^{\dim(X)/2}$ we can
consider the {\em Mukai flop} of $X$, obtained by blowing up the
projective space and blowing down the exceptional divisor 
$$E\simeq \bP(\Omega^1_{\bP^{\dim(X)/2}})$$
along the opposite ruling.  Our goal is to characterize possible homology
classes of such submanifolds, modulo the monodromy representation on
the cohomology of $X$.

Assuming $X$ contains a Lagrangian projective space $\bP^{\dim(X)/2}$,
let $\ell\in \rH_2(X,\bZ)$ denote the class of a line in $\bP^{\dim(X)/2}$, and 
$\lambda=N\ell\in \rH^2(X,\bZ)$
a positive integer multiple.  We can take $N$ to be the index of 
$\rH^2(X,\bZ) \subset \rH_2(X,\bZ)$.  Hodge theory \cite{Ran,Voisin} shows
that the deformations of $X$ containing a deformation of the Lagrangian
space coincide with the deformations of $X$ for which $\lambda \in \rH^2(X,\bZ)$
remains of type $(1,1)$.  Infinitesimal Torelli implies this is a divisor
in the deformation space, i.e., 
$$\lambda^{\perp} \subset \rH^1(X,\Omega^1_X) \simeq \rH^1(X,\cT_X).$$

We seek to establish intersection theoretic properties of $\ell$ for
various deformation-equivalence classes of holomorphic symplectic manifolds.
Previous results in this direction include
\begin{enumerate}
\item If $X$ is a K3 surface then $\left(\ell,\ell\right)=-2$.
\item If $X$ is deformation equivalent to the Hilbert scheme of length-two subschemes
of a K3 surface then $\left(\ell,\ell\right)=-5/2$. \cite{HTGAFA08}
\item If $X$ is deformation equivalent to a generalized Kummer fourfold
then $\left(\ell,\ell\right)=-3/2$.   \cite{HT10}
\end{enumerate}
Here we prove
\begin{theo} \label{theo:main}
Let $X$ be a six-dimensional K\"ahler manifold,
deformation equivalent to the
Hilbert scheme of length-three subschemes of 
a K3 surface.  Let $\bP^3 \subset X$ be a smooth subvariety and
$\ell \subset \bP^3$ a line.  Then $\left(\ell,\ell\right)=-3$
and $\rho=2\ell \in \rH^2(X,\bZ)$.  Furthermore, we have
$$\left[ \bP^3 \right]=\frac{1}{48}\left( \rho^3 + \rho^2c_2(X)\right).$$
\end{theo}
This uniquely characterizes the class of the Lagrangian plane,
modulo the monodromy action, which acts transitively on the
$\rho \in \rH^2(X,\bZ)$ with $\left(\rho,\rho\right)=-12$
and $\left(\rho,\rH^2(X,\bZ)\right)=2\bZ$ \cite[\S 3]{GHS}.  

In general, we conjectured in \cite{HT09} that if $X$ is of dimension $2n$ then
$\left(\ell,\ell\right)= -(n+3)/2$, if $X$ is deformation equivalent to a Hilbert
scheme of a K3 surface.
Our main motivation for making these conjectures is to achieve a classification
of extremal rational curves on irreducible holomorphic symplectic varieties
(i.e., generators of extremal rays of birational contractions) in terms of
intersection properties under the Beauville-Bogomolov form.  

The structure of this paper is as follows:  Section~\ref{sect:cohomology} reviews
the cohomology groups of Hilbert schemes of K3 surfaces;  Section~\ref{sect:ring}
focuses on the ring structure.  We employ representation theory to get results
on the Hodge classes in Section~\ref{sect:representation}.  The Hilbert scheme
of length-three subschemes is studied in detail in Section~\ref{sect:lengththree}.  
We extract the distinguished absolute Hodge class in the middle cohomology
in Section~\ref{sect:indecomp};  here `absolute Hodge classes' are
those that remain Hodge under arbitrary deformations of complex structure.
The computation of the class of the
Lagrangian three planes is worked out in Section~\ref{sect:LTP}, modulo
a number theoretic result.  This is proved in Section~\ref{sect:DA}.

{\bf Acknowledgments:} We are grateful to Noam Elkies, Lothar G\"ott\-sche, Manfred Lehn, Eyal Markman, and Christoph Sorger for useful conversations.
The second author was supported by National Science Foundation Grant 0554491 and 0901645;
the third author was supported by National Science Foundation
Grants 0554280 and 0602333.  We appreciate the hospitality of the American Institute of
Mathematics, where some of this work was done.

\section{Cohomology of Hilbert schemes}
\label{sect:cohomology}
Let $X$ be deformation equivalent to the punctual Hilbert scheme $S^{[n]}$,
where $S$ is a K3 surface. 
For $n>1$ the Beauville-Bogomolov form can be written
\cite[\S 8]{beauville}
$$\rH^2(X,\bZ) \simeq \rH^2(S,\bZ)_{\left(,\right)} \oplus_{\perp} \bZ \delta, \quad
				\left(\delta,\delta\right)=-2(n-1)$$
where $2\delta$ is the class of the `diagonal' divisor 
$\Delta^{[n]} \subset S^{[n]}$ parameterizing
nonreduced subschemes.  For each homology class $f\in \rH^2(S,\bZ)$, 
let $f \in \rH^2(X,\bZ)$ denote the class parameterizing subschemes with
some support along $f$.  This is compatible with the lattice embedding
above.  
Duality gives a $\bQ$-valued form 
on homology
$$
\rH_2(X,\bZ) \simeq \rH_2(S,\bZ)_{\left(,\right)} \oplus_{\perp} \bZ \delta^{\vee}, \quad
			\left(\delta^{\vee},\delta^{\vee}\right)=-\frac{1}{2(n-1)},$$
where $\delta^{\vee}$ is characterized as the homology class orthogonal
to $\rH^2(S,\bZ)$ and satisfying $\delta^{\vee}\cdot \delta =1$.

\begin{theo} \cite{Gott90}
Let $S$ be a K3 surface and $S^{[n]}$ its Hilbert scheme.  
Consider the Poincar\'e polynomial
$$p(S^{[n]},z)=\sum_{j=0}^{4n} \beta_j(S^{[n]})z^j.$$
Then 
$$\sum_{n=0}^{\infty} p(S^{[n]},z)t^n= \prod_{m=1}^{\infty}
(1-z^{2m-2}t^m)^{-1}(1-z^{2m}t^m)^{-22}(1-z^{2m+2}t^m)^{-1}.$$
\end{theo}

To save space, we write 
$$q(S^{[n]},z)=\sum_{j=0}^{n} \beta_{2j} z^j,$$
which determines the Poincar\'e polynomial by Poincar\'e duality.
We have
$$\begin{array}{rcl}
q(S,z)&=& 1+ 22z  \\
q(S^{[2]},z) &=& 1 + 23 z + 276 z^2 \\
q(S^{[3]},z) &=& 1 + 23 z + 299 z^2 + 2554 z^3.
\end{array}
$$

A theorem of Verbitsky \cite[Theorem 1.5]{Verb} asserts that the 
homomorphism arising from the cup product
$$\mu_{k,n}:\mathrm{Sym}^k \rH^2(S^{[n]},\bQ) \ra \rH^{2k}(S^{[n]},\bQ)$$
is injective for $k \le n$.  Thus its image has dimension
$$\binom{22+k}{k}.$$
In light of the computations above, $\mu_{2,2}$ is an isomorphism,
$\mu_{2,3}$ has cokernel of dimension $23$, and $\mu_{3,3}$
has cokernel of dimension 
$$2554-\binom{25}{3}=254=\binom{23}{2}+1.$$
The cup product also induces a homomorphism
$$\mathrm{coker}(\mu_{2,3}) \otimes \rH^2(S^{[3]},\bQ) \ra
 \mathrm{coker}(\mu_{3,3}).$$
This homomorphism has been observed by Markman \cite[p. 80]{MarkCrelle}.
More generally, he analyzes what classes are needed
to generate the cohomology ring $\rH^*(S^{[n]},\bQ)$, beyond those
coming $\rH^2(S^{[2]},\bQ)$.  Markman uses Chern classes of universal
sheaves over the product $S^{[n]} \times S$;  a detailed discussion of the 
$n=3$ case is given in \cite[Ex. 14]{MarkCrelle}.

\section{The ring structure on cohomology}
\label{sect:ring}

Lehn-Sorger \cite{LS} and Nakajima \cite{Nakajima} described 
$\rH^*(S^{[n]},\bQ)$ in terms of $\rH^*(S,\bQ)$.  
We review the Lehn-Sorger formalism for the cup product on the cohomology ring.

Let $S$ be a K3 surface and $A=\rH^*(S,\bQ)(1)$, the cohomology ring
shifted so that it has weights $-2,0$, and $2$;  this is written 
as $\rH^*(S,\bQ)[2]$ in their paper.  Shifting the weights
changes the sign of the intersection form, which is denoted by
$\left<,\right>$;  this has signature $(20,4)$.
Let $T:A \ra \bQ$ denote the linear form 
$$\gamma \mapsto -\int_S \gamma$$
and $\left<,\right>$ the induced bilinear form
$$\left<\gamma_1,\gamma_2\right>=T(\gamma_1\gamma_2)=-\int_S \gamma_1 \gamma_2.$$

For each $n \in \bN$, we endow $A^{\otimes n}$ with an analogous structure.  We shall
use the fact that $A$ has only graded pieces of {\em even} degrees to simplify the description
in \cite{LS}.  In this situation, graded commutative multiplication rules are in fact
commutative, given by the rule
$$(a_1 \otimes \cdots \otimes a_n)\cdot (b_1 \otimes \cdots \otimes b_n)=
(a_1b_1)\otimes \cdots \otimes (a_nb_n).
$$
The linear form 
$$T:A^{\otimes n} \ra \bQ$$
is defined by
$$T(a_1\otimes \cdots \otimes a_n) =T(a_1)\cdots T(a_n).$$
Let $\left<,\right>$ denote the associated bilinear form
$$\left<a,b\right>=T(a\cdot b).$$

The symmetric group $\fS_n$ acts on $A^{\otimes n}$ by the rule  
$$\pi(a_1 \otimes \cdots \otimes a_n)=
a_{\pi^{-1}(1)} \otimes \cdots \otimes a_{\pi^{-1}(n)}.$$

Given a partition $n=n_1+\ldots+n_k$ with $n_1,\ldots,n_k \in \bN$, we have
a generalized multiplication map
$$\begin{array}{rcl}
A^{\otimes n} & \ra & A^{\otimes k}\\
a_1\otimes \cdots \otimes a_n & \mapsto & (a_1\cdots a_{n_1})
\otimes  \cdots \otimes 
(a_{n_1+\cdots+n_{k-1}+1}\cdots a_{n_1+\cdots+n_k}).
\end{array}
$$

Given a finite set $I\subset \{1,\ldots,n\}$, let $A^{\otimes I}$ denote
the tensor power with factors indexed by elements of $I$.  Given a surjection
$\phi:I \ra J$, there is an induced multiplication
$$
\phi^* \colon A^{\otimes I} \ra  A^{\otimes J} 
$$
defined as above.  Let 
$$\phi_*: A^{\otimes J} \ra A^{\otimes I}$$
denote the {\em adjoint} of $\phi^*$, i.e., 
$$\left< \phi^*a,b \right>=\left<a,\phi_* b\right>$$
for $a \in A^{\otimes I}$ and $b\in A^{\otimes J}$. 

We have the composite 
$$A \stackrel{\Delta_*}{\ra} A\otimes A  \ra A,$$
where the first map is adjoint comultiplication and the second is
multiplication.  Let $e:=e(A)$ denote the image of $1$ under the 
composed map.
\begin{rema}
We evaluate the signs of $\Delta_*1$ and $e(A)$. Let $\Delta_S$ denote the fundamental
class of the diagonal in $\rH^*(S\times S,\bZ)=\rH^*(S,\bZ) \otimes \rH^*(S,\bZ)$.
Using the adjoint property, we have
$$\begin{array}{rcl}
\left<\Delta_*1, \alpha \otimes \beta \right> &=& \left< 1, \alpha\beta \right> \\
					      &=& T(\alpha\beta) \\
					      &=& -\int_S \alpha \beta 
\end{array}
$$
whereas 
$$\begin{array}{rcl}
\left<\Delta_S, \alpha \otimes \beta \right> &=& \left< \sum_j e_j \otimes e_j^{\vee}, \alpha\otimes \beta \right> \\
					     &=& \sum_j T(e_j \alpha) T(e_j^{\vee}\beta)\\
					     &=& \int_S \alpha \beta,
\end{array}
$$
where $\{e_j\}$ is a homogeneous basis for $\rH^*(S,\bQ)$ with Poincar\'e-dual basis $e_j^{\vee}$.  
Therefore, we find 
\begin{equation}
\Delta_*1=-[\Delta_S].  \label{eqnsignswitch}
\end{equation}
Furthermore, we have
$$\int_S e(A)=-T(e(A))=-\left<e(A),1\right>=-\left<\Delta_*1,\Delta_*1\right>=-\chi(S)=-24,$$
so $e(A)$ is a {\em negative} multiple of the point class.  Nevertheless, we still have
(cf. \cite[\S 2.2]{LS})
$$e(A)=\chi(S) \mathrm{vol}, \quad \text{ where } \quad T(\mathrm{vol})=1,$$
but $\mathrm{vol}$ differs from the standard volume form by sign.  
\end{rema}

Let $\left<\pi\right> \backslash [n]$ denote the set of orbits
of $[n]=\{1,2,\ldots,n\}$ under the action of $\pi$.  Set
$$A\{\fS_n\}=\oplus_{\pi\in \fS_n} A^{\otimes \left<\pi\right> \backslash [n]} \cdot \pi$$
which admits an action of $\fS_n$.  First, note that $\sigma \in \fS_n$ induces
a bijection
$$\begin{array}{rcl}
\sigma:\left<\pi\right> \backslash [n] & \ra & \left<\sigma \pi \sigma^{-1}\right>
\backslash [n] \\
x & \mapsto & \sigma x.
\end{array}
$$
Thus we obtain an isomorphism
$$\begin{array}{rcl}
\tilde{\sigma}: A\{S_n\} & \ra & A\{S_n\} \\
		a \pi & \mapsto & \sigma^* \sigma\pi \sigma^{-1}.
\end{array}
$$ 
\begin{exam} \cite[2.9, 2.17]{LS}
We have $A\{\fS_2\}=A^{\otimes 2}\mathrm{id} \oplus A (12)$ 
and
$$
A\{\fS_3\}=A^{\otimes 3}\mathrm{id}\oplus A^{\otimes 2}(12) \oplus
A^{\otimes 2}(13)\oplus A^{\otimes 2}(23) 
           \oplus A(123) \oplus A(132).
$$
\end{exam}

Let $A^{[n]}\subset A\{\fS_n\}$ denote the invariants under this action.    
Then we have \cite[\S 2]{LS}
$$A^{[n]}=\sum_{\| \alpha \|=n} \bigotimes_{i}\Sym^{\alpha_i}A,$$
where $\alpha$ corresponds to a partition
$$\underbrace{1+\cdots+1}_{\alpha_1 \text{ times}}+
\underbrace{2+\cdots+2}_{\alpha_2 \text{ times}}+\cdots$$
and
$$n=\|\alpha \|=\alpha_1+2\alpha_2+\cdots+n\alpha_n.$$
Note that this is compatible with Hodge structures;  in particular,
$A^{[n]}$ is a representation of the Hodge group of $S$ and the 
special orthogonal  group $G_S$ associated with the intersection form on $\rH^2(S,\bR)$.
We interpret this as acting on $A$, trivially on the summands $\rH^0(S,\bR)$
and $\rH^4(S,\bR)$.

\begin{theo} \cite[Theorem 3.2]{LS}
Let $S$ be a K3 surface.  Then there is a canonical isomorphism of graded rings
$$(\rH^*(S,\bQ)[2])^{[n]} \stackrel{\sim}{\ra}\rH^*(S^{[n]},\bQ)[2n].$$
\end{theo}

In the cohomology of the Hilbert scheme, the subring 
generated by $\rH^2(S^{[n]})$ plays a special role.  
We have an isomorphism
$$\rH^2(S^{[n]},\bZ) = \rH^2(S,\bZ) \oplus \bZ \delta,$$
where $2\delta$ parameterizes the non-reduced schemes of $S$.  We express this in terms
of our presentation.  Given $D \in \rH^2(S,\bZ)$, the class
$$\sum_{i=1}^n 1_{\{1\}} \otimes \cdots \otimes 1_{\{i-1\}} \otimes D_{\{i\}}
\otimes 1_{\{i+1\}}\otimes \cdots \otimes 1_{\{n\}} (\mathrm{id})$$
is the corresponding class in $\rH^2(S^{[n]},\bQ)[2n]$.  
Using the explicit form of the isomorphism in \cite[2.7]{LS}
and Nakajima's isomorphism (\cite[Thm. 3.6]{LS}), we find that 
$$\delta=\sum_{1 \le i<j \le n} 1_{\{1\}} \otimes \ldots \otimes 1_{\{i-1\}} \otimes 1_{\{i,j\}} 
\otimes 1_{\{i+1\}} \otimes \cdots \otimes 1_{\{j-1\}} \otimes 1_{\{j+1 \}} \otimes \cdots \otimes 1_{\{n\}}(ij).$$
Here is the essence of the computation:
the interpretation of the nonreduced subschemes via the correspondence
$$Z_2=\{(\xi,x,\xi'):|\xi'|-|\xi|=2x \} \subset S^{[n-2]} \times S \times S^{[n]}$$
allows us to express $\delta$ in terms of Nakajima's creation and annihilation
operators, and thus in 
$$\rH^*(S^{[n]},\bQ)[2n].$$

We describe the general rule for evaluating the fundamental class in 
$A^{[n]}$.  Let 
$$[\mathrm{pt}]\in \rH^4(S,\bZ)[2] \subset A$$
be the point class, which is of degree $-2$.  Let
$$[\mathrm{pt}]_{\{1\}} \otimes \cdots \otimes [\mathrm{pt}]_{\{n\}} (\mathrm{id})
\in A^{[n]}$$
denote the unique class of degree $-2n$ up to scalar.  Then the class of a point in
$S^{[n]}$ is equal to \cite[2.10]{LS}
\begin{equation} \label{eqn:pointclass}
[\mathrm{pt}_{S^{[n]}}]=\frac{1}{n!}
[\mathrm{pt}]_{\{1\}} \otimes \cdots \otimes [\mathrm{pt}]_{\{n\}} (\mathrm{id})
.
\end{equation}

\section{Decomposition of the cohomology representation} 
\label{sect:representation}

We summarize general results from representation theory.  
For an orthogonal group of odd dimension $2r+1$, the highest weights 
$\lambda=(\lambda_1,\ldots,\lambda_r)$ of irreducible representations $V(\lambda)$
are vectors consisting entirely of integers (or half integers) in the fundamental chamber
$$\{\lambda_1 \ge \lambda_2 \ge \ldots \ge \lambda_{r-1} \ge \lambda_r \ge 0 \}.$$
Since we only consider even-weight representations, we ignore cases
where the $\lambda_j$ are half-integers.  For orthogonal groups of even 
dimension $2r$, the fundamental chamber is
$$\{\lambda_1 \ge \lambda_2 \ge \ldots \ge \lambda_{r-1} \ge |\lambda_r| \ge 0 \}.$$
Recall that
\begin{itemize}
\item{$V(1,0,\ldots)$ is the standard representation $V$.}
\item{We have
$$V(\underbrace{1,\cdots,1}_{k \text{ times}},0,\cdots)=\bigwedge^k V,$$
provided $k < r$ (in the
even case) or 
$k\le r$ (in the odd case);
see, for instance, \cite[Thms. 19.2 and 19.14]{FulHar}.}
\item{$V(k,0,\ldots)=\Sym^k(V)/\Sym^{k-2}(V)$, embedded via the dual to the quadratic form
on $V$.}
\item{For the odd orthogonal group, we have
$$\dim V(\lambda)=\prod_{i<j} \frac{\ell_i-\ell_j}{j-i} \prod_{i \le j}\frac{\ell_i+\ell_j}{2n+1-i-j}$$
where $\ell_i=\lambda_i+n-i+\frac{1}{2}$ \cite[p. 408]{FulHar}.}
\item{For the even orthogonal group, we have
$$\dim V(\lambda)=\prod_{i<j} \frac{\ell^2_i-\ell^2_j}{(j-i)(2n-i-j)}$$
where $\ell_i=\lambda_i+n-i$ \cite[p. 410]{FulHar}.}
\item{Let $V_X(\lambda)$ denote an irreducible representation of an orthogonal
group $G_X$ of dimension $2r+1$, $G_S \subset G_X$ the
orthogonal subgroup $G_S\subset G_X$ of dimension $2r$ fixing a non-isotropic vector
with negative self-intersection,
and $V_S(\overline{\lambda})$ the representation of $G_S$ with highest weight $\overline{\lambda}$.  
Then we have the branching rule \cite[p. 426]{FulHar}
$$\mathrm{Res}^{G_X}_{G_S}V_X(\lambda)= \oplus_{{\overline \lambda}} V_S(\overline \lambda),$$
where the sum ranges over all $\overline{\lambda}$ with
$$\lambda_1 \ge {\overline \lambda_1} \ge \lambda_2 \ge {\overline \lambda_2} \ge \cdots 
\lambda_r \ge |{\overline \lambda_r}|.$$}
\end{itemize}

Let $X$ be a generic deformation of $S^{[n]}$.  Our goal is to decompose
$\rH^*(X,\bQ)$ into irreducible representations for the action of the
identity component $G_X$ of the special orthogonal group 
associated with the Beauville-Bogomolov form on $\rH^2(X,\bQ)$. 
Let $G_S$ denote the identity component of the special orthogonal
group associated with the intersection form on $\rH^2(S,\bQ)$.  
The decomposition 
$$\rH^2(S^{[2]},\bZ)=\rH^2(S,\bZ) \oplus_{\perp} \bZ \delta$$
induces an inclusion $G_S \subset G_X$.  

\begin{prop}
Let $X$ be deformation 
equivalent to $S^{[n]}$ for some $n$.  
Then $G_X$ admits a representation on the cohomology ring of $X$.  
\end{prop}
\begin{proof}
Let $\mathrm{Mon}\subset \mathrm{Aut}(\rH^*(X,\bZ))$ denote the monodromy group,
i.e., the group generated by the monodromy representations of all connected families
containing $X$.  
Let $\mathrm{Mon}^2\subset \mathrm{Aut}(\rH^2(X,\bZ))$ denote its image under
projection to the second cohomology group, so we have an exact sequence
$$ 1 \ra K \ra \mathrm{Mon} \ra \mathrm{Mon}^2 \ra 1.$$
Markman has shown \cite[\S 4.3]{Mark1} that $K$ is finite.  

Note that $G_X$ is a connected component of the Zariski closure of $\mathrm{Mon}^2$
(see, for example \cite[\S 1.8]{Mark1}).  Since $\mathrm{Mon}$ and $\mathrm{Mon}^2$
differ only by finite subgroups, it follows that the universal cover $\widetilde{G_X}\ra G_X$
acts on the cohomology ring of $X$.  Since the cohomology of $X$ is nonzero only
in even degrees, this representation passes to $G_X$.  
\end{proof}

In principle, we can decompose $\rH^*(X,\bR)$ explicitly into isotypic
components as follows:
\begin{enumerate}
\item{Fix an embedding $G_S \subset G_X$, e.g., using the isomorphism
$$\rH^2(X,\bZ)\simeq \rH^2(S,\bZ) \oplus_{\perp} \bZ \delta,$$
and compatible maximal tori (both of which have rank $11$).}
\item{Identify the highest-weight irreducible $G_S$-representation 
$V_{S}(\lambda) \subset \rH^*(S^{[n]},\bR)$, which is a summand
of the restriction of an irreducible $V_{X}(\lambda) \subset \rH^*(X,\bR)$.  
Decompose $V_X(\lambda)$ into irreducible $G_S$-representations.}
\item{Repeat step two for $\rH^*(X,\bR)/V_{X}(\lambda)$ and subsequent quotients.}
\end{enumerate}

First consider $X=S^{[2]}$.  We have decompositions
$$\rH^*(S^{[2]})=A \oplus \Sym^2(A)$$
inducing
$$\begin{array}{rcl}
\rH^2(S^{[2]})&=& \rH^0(S) \oplus (\rH^0(S)\otimes \rH^2(S))={\bf 1}_S \oplus V_S(1,0,\ldots) \\
\rH^4(S^{[2]})&=& \rH^2(S) \oplus (\rH^0(S)\otimes \rH^4(S)) \oplus \Sym^2(\rH^2(S)) \\
	      &=& V_S(1,0,\ldots) \oplus {\bf 1}_S^{\oplus 2} \oplus V_S(2,0,\ldots)
\end{array}
$$
Let $V_{X}(2,0,\ldots,0)$ denote the highest-weight representation associated
to $\Sym^2(\rH^2(X))$ so that
$$\Sym^2(\rH^2(X))=V_{X}(2,0,\ldots) \oplus {\bf 1}_X.$$
The branching rule gives 
$$V_{X}(1,0,\ldots)= V_S(1,0,\ldots) \oplus {\bf 1}_S$$
and
$$V_{X}(2,0,\ldots)= V_S(2,0,\ldots) \oplus V_S(1,0,\ldots) \oplus 
{\bf 1}_S.$$
Therefore we obtain
$$\begin{array}{rcl}
\rH^2(X) &=& V_X(1,0,\ldots) \\
\rH^4(X) &=& V_X(2,0,\ldots)\oplus {\bf 1}_X. 
\end{array}
$$

Now consider $X=S^{[3]}$.
We have 
$$\rH^*(S^{[3]})= A \oplus (A\otimes A) \oplus \Sym^3(A)$$
inducing following decompositions (as described in \cite[Example 2.9]{LS}):
$$
\begin{array}{rcl}
\rH^2(S^{[3]}) &=& (\rH^0(S)^{\otimes 2}) \oplus (\rH^2(S)\otimes \rH^0(S)^{\otimes 2})  \\
		&=& {\bf 1}_S \oplus V_S(1,0\ldots) \\
\rH^4(S^{[3]}) &=& \rH^0(S) \oplus (\rH^0(S)\otimes \rH^2(S))^{\oplus 2} \\
& &\oplus (\Sym^2(\rH^2(S))\otimes \rH^0(S)) \oplus (\rH^4(S)\otimes \rH^0(S)^{\otimes 2}) \\
               &=& {\bf 1}_S^{\oplus 3} \oplus V_S(1,0,\ldots)^{\oplus 2} \oplus V_S(2,0,\ldots) \\ 
\rH^6(S^{[3]}) &=& \rH^2(S) \oplus (\rH^2(S)\otimes \rH^2(S))
  \oplus 
(\rH^0(S) \otimes \rH^4(S))^{\oplus 2} \\
& &\oplus \Sym^3(\rH^2(S)) \oplus (\rH^4(S) \otimes \rH^2(S) \otimes \rH^0(S)) \\
		&=& {\bf 1}_S^{\oplus 3} \oplus V_S(1,0,\ldots)^{\oplus 3} \oplus V_S(1,1,0,\ldots)\\ 
		& & \oplus V_S(2,0,\ldots) \oplus V_S(3,0,\ldots). 
\end{array}
$$
Let $V_{X}(1,1,0,\ldots)=\bigwedge^2 V_X(1,0,\ldots)$ and $V_X(3,0,\ldots)$
denote the highest weight representation in $\Sym^3(V_X(1,0,\ldots))$ so that
$$\Sym^3(V_X(1,0,\ldots))=V_X(3,0,\ldots) \oplus V_X(1,0,\ldots).$$
Therefore we obtain
$$\begin{array}{rcl}
\rH^2(X) &=& V_X(1,0,\ldots) \\
\rH^4(X) &=& V_X(2,0,\ldots) \oplus V_X(1,0,\ldots) \oplus {\bf 1}_X \\
\rH^6(X) &=& V_X(3,0,\ldots) \oplus V_X(1,1,0\ldots)\oplus V_X(1,0,\ldots) \oplus 
		{\bf 1}_X.
\end{array}
$$
The trivial factor in $\rH^4(X)$ corresponds to the Chern class $c_2(X)$;
our main task is to analyze the trivial factor in $\rH^6(X)$.

\section{Cohomology computations for length-three subschemes}
\label{sect:lengththree}

The general rule for multiplication in $A\{\fS_n\}$ is fairly complicated, so we will 
only give a formula in the case ($n=3$) we need.  The fact that $A$ only
has terms of even degree simplifies the expressions of \cite[2.17]{LS}:  

$$\begin{array}{rcl}
(\alpha_{\{1,2\}}\otimes \beta_{\{3\}})(12) \cdot 
(\gamma_{\{1,3\}} \otimes \delta_{\{2\}})(13) &=&
\alpha \beta \gamma \delta (132) \\
(\alpha_{\{1,2\}} \otimes \beta_{\{3\}})(12) \cdot
(\gamma_{\{1,2\}} \otimes \delta_{\{3\}})(12) &=&
\Delta_*(\alpha \gamma) \otimes (\beta \delta) (\mathrm{id}) \\
\alpha_{\{1,2,3\}} (123) \cdot \beta_{\{1,2,3\}}(123) &=&
 (\alpha \beta e)(132) \\
\alpha_{\{1,2,3\}}(123) \cdot \beta_{\{1,2,3\}}(132) &=& 
(\Delta_*(\alpha \beta))_{\{1,2,3\}} (\mathrm{id}),
\end{array}
$$
where $\Delta_*$ is the adjoint of the threefold multiplication
$A\otimes A \otimes A \ra A$.  

The remaining products can be deduced as formal consequences using the associativity
of the multiplication, e.g.,
$$\begin{array}{l}
(\alpha_{\{1,2\}} \otimes \beta_{\{3\}})(12) \cdot \gamma_{\{1,2,3\}}(132) \\
\quad = (\alpha_{\{1,2\}} \otimes \beta_{\{3\}})(12) \cdot (\gamma_{\{1,2\}} \otimes
1_{\{3\}})(12)\cdot (13) \\
\quad = (\Delta_*(\alpha \gamma)_{\{1,2\}} \otimes \beta_{\{3\}})(\mathrm{id}) \cdot
(1_{\{1,3\}} \otimes 1_{\{2\}})(13) \\
\quad = \alpha \beta \gamma (\Delta_*(1))_{\{1,3\}, \{2\}} (13),
\end{array}
$$
where $\alpha,\beta,$ and $\gamma$ act on the diagonal via either the first or second
variable.  Thus in particular
$$(12)\cdot (132)=(\Delta_*(1))_{\{1,3\},\{2\}}(13).$$  

We compute intersections among the absolute Hodge classes for $S^{[3]}$, 
i.e., classes that are Hodge for general K3 surfaces $S$.  From now on,
to condense notation we omit factors of the form $1_{\{i\}},1_{\{i,j\}}$, etc.
from our expressions.  

Based on the representation-theoretic analysis in Section~\ref{sect:representation},
we expect one independent classes in codimension one,
three in codimension two, and three in codimension three.
We have the unique divisor
$$\delta =  (12)+(13)+(23).$$
In codimension two, we have 
$$\begin{array}{rcl}
P &=& [\mathrm{pt}]_{\{1\}}+ [\mathrm{pt}]_{\{2\}} + [\mathrm{pt}]_{\{3\}} \\
Q &=& \sum_{j=1}^{22} 
{e_j}_{\{1\}} \otimes {e^{\vee}_j}_{\{2\}} +
{e_j}_{\{1\}} \otimes {e^{\vee}_j}_{\{3\}} +
{e_j}_{\{2\}} \otimes {e^{\vee}_j}_{\{3\}}  \\
R &=& (132)+(123).
\end{array}
$$
In codimension three, we have
$$\begin{array}{rcl}
U &=& [\mathrm{pt}]_{\{1,2\}} (12) + 
[\mathrm{pt}]_{\{1,3\}} (13) + 
[\mathrm{pt}]_{\{2,3\}} (23) \\
V &=& [\mathrm{pt}]_{\{3\}} (12) + 
[\mathrm{pt}]_{\{2\}} (13) + 
[\mathrm{pt}]_{\{1\}} (23) \\
W &=& \sum_{j=1}^{22} {e_j}_{\{1,2\}} \otimes {e_j^{\vee}}_{\{3\}} (12)  +
{e_j}_{\{1,3\}} \otimes {e_j^{\vee}}_{\{2\}} (13) +
{e_j}_{\{2,3\}} \otimes {e_j^{\vee}}_{\{1\}} (23).
\end{array}
$$
Thus we have 
$$\begin{array}{rcl}
\delta^2&=& (\Delta_*1)_{\{1,2\}}(12)+ (\Delta_*1)_{\{1,3\}}(13)+(\Delta_*1)_{\{2,3\}}(23)
				\\
        & & \mathbin{+} \,\, 3((132)+(123)) \\
	&=&-2P-Q+3R.
\end{array}
$$
Furthermore, we have
$$\begin{array}{rcl}
\delta \cdot P  &=& ((12)+(13)+(23))\cdot
	 ([\mathrm{pt}]_{\{1\}} + [\mathrm{pt}]_{\{2\}} + [\mathrm{pt}]_{\{3\}}) \\
		&=& 2U+V \\
\delta \cdot Q  &=& ((12)+(13)+(23))\cdot (\sum_{j=1}^{22}
{e_j}_{\{1\}} \otimes {e^{\vee}_j}_{\{2\}} +
{e_j}_{\{1\}} \otimes {e^{\vee}_j}_{\{3\}} \\
& & \quad \quad +\,\, {e_j}_{\{2\}} \otimes {e^{\vee}_j}_{\{3\}} ) \\
                &=& 22([\mathrm{pt}]_{\{1,2\}}(12) +[\mathrm{pt}]_{\{1,3\}}(13) + [\mathrm{pt}]_{\{2,3\}}(23)) \\
                & & +2 ( \sum_{j=1}^{22} {e_j}_{\{1,2\}} \otimes {e_j^{\vee}}_{\{3\}}  +
{e_j}_{\{1,3\}} \otimes {e_j^{\vee}}_{\{2\}} +
{e_j}_{\{2,3\}} \otimes {e_j^{\vee}}_{\{1\}}) \\
                &=& 22U + 2W \\
\delta \cdot R &=&((12)+(13)+(23))((132)+(123))\\
	       &=&2({\Delta_*1}_{\{1,2\},\{3\}}(12)+{\Delta_*1}_{\{1,3\},\{2\}}+{\Delta_*1}_{\{2,3\},\{1\}})\\
               &=& -2(U+V+W).
\end{array}
$$
We deduce then that
$$\delta^3=\delta(-2P-Q+3R)=-32U-8V-8W.$$

Finally, we compute the intersection pairing on the subspace of the middle cohomology spanned
by $U,V,$ and $W$.  Dimensional considerations give vanishing
$$U^2=V^2=U\cdot W=V\cdot W=0.$$
For the remaining numbers, we get
$$\begin{array}{rcl}
U\cdot V &=& 
([\mathrm{pt}]_{\{1,2\}} (12) + 
[\mathrm{pt}]_{\{1,3\}} (13) + 
[\mathrm{pt}]_{\{2,3\}} (23) ) \\
& & \quad \quad \cdot ([\mathrm{pt}]_{\{3\}} (12) +
[\mathrm{pt}]_{\{2\}} (13) +
[\mathrm{pt}]_{\{1\}} (23)) \\
&=&-3 [\mathrm{pt}]_{\{1\}} \otimes [\mathrm{pt}]_{\{2\}} \otimes [\mathrm{pt}]_{\{3\}} \mathrm{id} 
\end{array}
$$
and 
$$\begin{array}{rcl}
W^2 &=&
(\sum_{j=1}^{22} {e_j}_{\{1,2\}} \otimes {e_j^{\vee}}_{\{3\}} (12)  +
{e_j}_{\{1,3\}} \otimes {e_j^{\vee}}_{\{2\}} (13) +
{e_j}_{\{2,3\}} \otimes {e_j^{\vee}}_{\{1\}} (23))^2 \\
    &=& -3\cdot 22  \cdot 
 [\mathrm{pt}]_{\{1\}} \otimes [\mathrm{pt}]_{\{2\}} \otimes [\mathrm{pt}]_{\{3\}} \mathrm{id}.
\end{array}
$$

\begin{rema}
As a consistency check, we evaluate
$$ \begin{array}{rcl}
\delta^6&=&(-32U-8V-8W)^2=2^6 (8UV+W^2)\\
  &=&2^6 (-24-66)
 [\mathrm{pt}]_{\{1\}} \otimes [\mathrm{pt}]_{\{2\}} \otimes [\mathrm{pt}]_{\{3\}} \mathrm{id}.
\end{array}
$$
Using the formula for the point class (Equation~\ref{eqn:pointclass}), we obtain
$$\delta^6=-\frac{2^7 \cdot 3^2 \cdot 5}{2\cdot 3}=-2^6 \cdot 3 \cdot 5.$$
This is compatible with the Fujiki-type identity
$$D^6=15\left(D,D\right)^3, \quad D\in H^2(S^{[3]},\bQ),$$
as $\left(\delta,\delta\right)=-4$.  
\end{rema}

\section{Evaluation of the distinguished absolute Hodge class}
\label{sect:indecomp}

Let $S$ be a general K3 surface and $X$ a general deformation of $S^{[3]}$.  
The computations above show that the middle cohomology of $X$ admits
one Hodge class 
$$\rH^6(X,\bQ)\cap \rH^{3,3}(X)=\bQ \eta$$  
and the middle cohomology of $S^{[3]}$ admits three Hodge classes
$$\rH^6(S^{[3]},\bQ) \cap \rH^{3,3}(S^{[3]})= \bQ \eta \oplus \bQ \delta^3 \oplus  \bQ c_2(X)\delta.$$
Our goal is to compute the self-intersection of $\eta$, at least up to the square of a rational number.   
Note that $\eta$ is orthogonal to $\delta^3$ and $\delta c_2(X)$ under the intersection form,
by the analysis in Section~\ref{sect:representation}.  
The analysis here gives the one structure constant left open in \cite[Ex. 14]{MarkCrelle}.  

\begin{prop} \label{prop:eta}
Let $X$ be deformation equivalent to $S^{[3]}$, for $S$ a K3 surface.
Let $\eta\in \rH^6(X,\bQ)$ 
denote the unique (up to scalar) 
absolute Hodge class.
Then $\eta^2=-3\cdot 443$.  
\end{prop}
\begin{proof}
The argument relies heavily on the analysis in Section~\ref{sect:lengththree}.
We extract the decomposable classes in codimension three.  We have $\delta^3$ already and
$$\delta \cdot P=2U+V.$$
Hence the subspace $\mathrm{span}\{2U+V,V-W \}$ is spanned by decomposable classes and has orthogonal 
complement spanned by $2U-V+11W$.  Thus we have 
$$\eta=2U-V+11W$$
and 
$$\begin{array}{rcl}
\eta^2&=&-4UV+121W^2 \\
      &=&(12-121 \times 66)([\mathrm{pt}] \otimes [\mathrm{pt}] \otimes 
[\mathrm{pt}] )\mathrm{id} \\
      &=&-3 \cdot 443. 
\end{array}
$$
\end{proof}

\section{Proof of the main theorem} \label{sect:LTP}

We compute the cohomology class of a Lagrangian subspace $\bP^3 \subset X$,
where $X$ is deformation equivalent to the Hilbert scheme of length
three subschemes.  As we shall see, the formula for $[\bP^3]$ involves
only decomposable classes, and not the absolute Hodge class $\eta$:

\begin{lemm}
Let $\bP^n \subset X$ be embedded in a general irreducible holomorphic
symplectic variety of dimension $2n$.  Then we have
$$c_{2j}(\cT_X|\bP^n)=(-1)^jh^{2j}\binom{n+1}{j},$$
where $h$ is the hyperplane class.
\end{lemm}
This is proved using the exact sequence
$$0 \ra \cT_{\bP^n} \ra \cT_X|\bP^n \ra \cN_{\bP^n/X} \ra 0$$
and 
$$\cN_{\bP^n/X}\simeq \cT_{\bP^n}^{\vee},$$
reflecting the fact that $\bP^n$ is a Lagrangian subvariety of $X$.

Regarding 
$$
\rH^2(X,\bZ) \subset \rH_2(X,\bZ)
$$
as a subgroup of index four, we can express
$\ell=\lambda/4$ for some divisor class $\lambda \in \rH^2(X,\bZ)$.  
(This might not be primitive.)

Given a deformation of $X$ such that $\lambda$ remains algebraic,
the subvariety $\bP^3$ deforms as well \cite{HTGAFA99}.  
Without loss of generality, we can deform $X$ to a variety containing 
a $\bP^3$, but otherwise having a general Hodge structure.  In particular,
we have a injection 
$$\Sym(\rH^2(X,\bQ)) \hookrightarrow \rH^*(X,\bQ).$$
We expect to be able to write
$$\left[ \bP^3 \right]=a\lambda c_2(X) + b \lambda^3 + d \eta$$
for some $a,b,d \in \bQ$.  

Furthermore, the Fujiki relations
\cite{Fujiki} imply that for each $f\in \rH^2(X,\bZ)$,
$$f^6= e_0 \left(f,f\right)^3, \quad
c_2(X)f^4= e_2 \left(f,f\right)^2, \quad
c_4(X)f^2 = e_4 \left(f,f\right)
$$
for suitable rational constants $e_0,e_2,e_4$.  
Precisely, we have \cite{EGL}
$$c_2^2(X)f^2=\frac{5}{2}c_4(X)f^2.$$

The Riemann-Roch formula gives
$$\chi(\cO_X(f))=\frac{f^6}{6!}+ \frac{c_2(X)f^4}{12 \cdot 4!}+
\frac{f^2(3c_2^2-c_4)}{720 \cdot 2!}+4.$$
On the other hand, we know that 
$$\chi(\cO_X(f))=\frac{1}{3!2^3}(\left(f,f\right)+8)(\left(f,f\right)+6)
					(\left(f,f\right)+4).$$
Perhaps the quickest way to check this formula is to observe that
if $X=S^{[3]}$ and $f$ is a very ample divisor on $S$ with
no higher cohomology then the induced sheaf $\cO_X(f)$ has
no higher cohomology and 
$$\dim \Gamma(\cO_X(f))=\dim \Sym^3(\Gamma(\cO_S(f)))=
\binom{\chi(\cO_S(f))+2}{3}.$$
Equating coefficients, we find
$$\begin{array}{rcl}
f^6 &=& 15 \left(f,f\right)^3 \\
f^4c_2 &=& 108 \left(f,f\right)^2 \\
f^2c_4 &=& 480 \left(f,f\right) \\
f^2c_2^2 &=& 1200 \left(f,f\right)
\end{array}
$$

We generate Diophantine equations for $a,b,\left(\lambda,\lambda\right)$
and eventually, $d$.
First, observe that
$$\left(\lambda,\ell \right)=\lambda\cdot \ell =\deg \lambda|\bP^3$$
so that $\lambda|\bP^3$ is $\left(\lambda,\lambda \right)/4$ times
the hyperplane class.  Thus we have
$$\left[\bP^3 \right]\lambda^3 = (\left(\lambda,\lambda \right)/4)^3$$
and
$$\left[\bP^3 \right]\lambda^3 = a\lambda^4c_2(X)+b\lambda^6.$$
Equating these expressions and evaluating the terms, we find
$$\left(\lambda,\lambda\right) (15 b -1/64)+108a=0.$$
We have divided out by $\left(\lambda,\lambda \right)$;  the
solution $\left(\lambda,\lambda\right)=0$ is not possible 
for geometric reasons, and we shall
exclude it algebraically below. 

Second, the Lemma on restrictions of Chern classes implies
$$\left[\bP^3 \right] \lambda c_2(X)=-\left(\lambda,\lambda\right)$$
whereas the formula for the class of $\bP^3$ yields
$$\left[\bP^3 \right]\lambda c_2(X)=a\lambda^2c_2(X)^2 + b \lambda^4 c_2(X).$$
Thus we obtain
$$108 b \left(\lambda,\lambda \right) + (1200 a  + 1)=0.$$

\begin{rema}
The cup product of $\rH^*(X)$ is compatible with the $G_X$-action, so the
subring generated by Chern classes and elements of $\rH^2(X)$ is orthogonal
to $\eta$.  Thus even if the decomposition of $[\bP^3]$
were to involve $\eta$, the computations up to this point would not reflect this.
\end{rema}

Finally, the fact that 
$$ \left[\bP^3 \right]^2=c_3(\cN_{\bP^3/X})=c_3(\cT_{\bP^3}^{\vee})=-4$$
yields the {\em cubic} equation
$$15b^2 \left(\lambda,\lambda\right)^3 + 216 ab \left(\lambda,\lambda\right)^2
+ 1200 \left(\lambda,\lambda \right)a^2 + d^2 \eta \cdot \eta =-4.$$
Proposition~\ref{prop:eta} implies that $\eta \cdot \eta = -11\cdot 443$.  
In particular, $\left(\lambda,\lambda\right)=0$ is excluded.

Eliminating $a$ and $b$ from these equations and setting $L=\left(\lambda,\lambda\right)$,
we obtain
\begin{equation} \label{eqn:elliptic}
2^{14}\cdot 3^2\cdot 11\cdot 443 d^2 = 5^2 L^3 + 2^5\cdot 3^2 L^2 + 2^8\cdot 5 L + 2^{16}\cdot 3\cdot 11.
\end{equation}
We know, {\em a priori}, that $L \in \bZ$ and $d\in \bQ$.  

\begin{prop}
The only solution to (\ref{eqn:elliptic}) with $L \in \bZ$ and $d\in \bQ$ is $d=0$ and $L=-48$.
\end{prop}
We assume this for the moment;  its proof can be found in Section~\ref{sect:DA}.

Back-substitution yields
$$a=1/96, \quad b=1/384, \quad \left(\ell,\ell\right)=-3.$$
We claim that $\lambda/2 \in \rH^2(X,\bZ)$, i.e., $\lambda$ is not
primitive.  Using the isomorphism
$$\rH_2(X,\bZ)=\rH_2(S,\bZ) \oplus_{\perp} \bZ \delta^{\vee}, \quad
\left(\delta^{\vee},\delta^{\vee}\right)=-1/4$$
we can express
$$\ell=D+m\delta^{\vee}, \quad D \in \rH_2(S,\bZ),m \in \bZ.$$
If $\lambda$ were primitive then $m$ would have to be odd and
$$-3=\left(\ell,\ell\right)=\left(D,D\right) - m^2 /4.$$
Since $\left(D,D\right) \in 2\bZ$, we have a contradiction.

\section{Diophantine analysis}
\label{sect:DA}

\begin{theo}
The only solution to
 \[ 2^{14}\dt 3^2\dt 11\dt443 d^2 = 5^2 L^3 + 2^5\dt 3^2 L^2 + 2^8\dt 5 L + 2^{16}\dt3\dt11 \]
with $L \in \bZ$ and $d \in \bQ$ is $L = -48$, $d = 0$.
\end{theo}

\begin{proof}
Put $x = 2^{-4}\dt5^2\dt11\dt443 (L + 48)$ and $y = 2\dt3\dt5^2\dt11^2\dt443^2d$. The equation then takes the form
\begin{equation}
\label{eq:weierstrass}
 E: y^2 = x^3 + ax^2 + bx
\end{equation}
where
 \[ a = -3^2\dt11\dt23\dt443, \qquad b = 2^2\dt5^2\dt11^3\dt13\dt443^2. \]
It suffices to prove the stronger statement that there are no solutions to \eqref{eq:weierstrass} with $x, y \in \bZ[\frac12]$, apart from $x = y = 0$.

The proof is given in two steps. Proposition \ref{prop:MW} below determines explicitly the structure of the Mordell--Weil group $E(\bQ)$. Proposition \ref{prop:integral} then identifies the integral points.
\end{proof}

Algorithms for both of these steps are implemented in computer algebra systems such as {\tt Sage} \cite{sage-4.4.1} and {\tt Magma} \cite{magma}, and the theorem may be verified this way. To avoid depending on the correctness of these systems, we give alternative proofs that use as little machine assistance as possible. The only step that is perhaps unreasonable to verify by hand is that a certain point $P$ with large coordinates (about $30$ digits) lies in $E(\bQ)$.

We first set notation and briefly recall some facts about point multiplication on elliptic curves. Let $O$ denote the zero element of $E(\bQ)$ (the point at infinity). For nonzero $R \in E(\bQ)$ we write
 \[ R = (x(R), y(R)) = \left( \frac{\alpha(R)}{e(R)^2}, \frac{\beta(R)}{e(R)^3} \right), \]
where $\alpha, \beta, e \in \bZ$, $e \geq 1$ and $(\alpha, e) = (\beta, e) = 1$.

Let $R \in E(\bQ)$, $R \neq O$. If $p$ is a prime, then $p \divides e(R)$ if and only if $R$ reduces to the identity in $E(\bF_p)$. If $m \geq 1$ and $mR \neq O$, then $e(R) \divides e(mR)$. For $m = 2$ we have the following formula:
\begin{equation}
\label{eq:double}
 x(2R) = \frac{\alpha(2R)}{e(2R)^2} = \frac{(\alpha(R)^2 - b \dt e(R)^4)^2}{4e(R)^2\big(\alpha(R)^3 + a \dt \alpha(R)^2 e(R)^2 + b \dt \alpha(R) e(R)^4\big)}.
\end{equation}
Moreover, if $R$ reduces to a nonsingular point in $E(\bF_p)$, then $p$ cannot divide both the numerator and denominator of the fraction on the right side of \eqref{eq:double}. In other words, there is no cancellation locally at $p$. One proof of this is given in \cite[Prop.~IV.2]{Wut-thesis}; as pointed out in that paper, it can also be proved from properties of real-valued non-archimedean local heights.

The discriminant of the Weierstrass equation \eqref{eq:weierstrass} is given by
 \[ \Delta = 16b^2(a^2 - 4b) = -2^8 \dt 5^4 \dt 11^8 \dt 13^2 \dt 113 \dt 127 \dt 443^6, \]
so the model is minimal, and the primes of bad reduction are $2$, $5$, $11$, $13$, $113$, $127$ and $443$. For $p = 2, 5, 11, 13, 443$, we have that $p \divides \alpha(R)$ if and only if $R$ reduces to a singular point of $E(\bF_p)$, i.e.~the only singular point of $E(\bF_p)$ is $(0:0:1)$ for these primes. The point $Q = (0, 0)$ has order two, and addition with $Q$ is given by the formula
\begin{equation}
\label{eq:add-Q}
 R + Q = \left(\frac{b}{x(R)}, \frac{-b \dt y(R)}{x(R)^2}\right).
\end{equation}

\begin{prop}
\label{prop:MW}
We have $E(\bQ) \cong \bZ \times (\bZ/2\bZ)$, where the free part is generated by the point $P$ with coordinates
 \[ \left(\frac{2 \dt 3^2 \dt 11^2 \dt 83^2 \dt 443^2 \dt 6481^2}{7^4 \dt 41^2 \dt 71^2 \dt 193^2}, \frac{2 \dt 3 \dt 11^3 \dt 31 \dt 83 \dt 163 \dt 443^2 \dt 6481 \dt 240623 \dt 3691717}{7^6 \dt 41^3 \dt 71^3 \dt 193^3} \right) \]
and the torsion part by $Q = (0, 0)$.
\end{prop}
\begin{proof}
We first check that the torsion subgroup is as described. We have $E(\bF_3) = \bZ/2\bZ \times \bZ/2\bZ$ and $E(\bF_{19}) = \bZ/2\bZ \times \bZ/7\bZ$. For $\ell$ prime, by \cite[Prop.~VII.3.1]{Sil-AEC} we see that $E(\bQ)[\ell]$ injects into $E(\bF_3)$ for $\ell \neq 3$ and that $E(\bQ)[\ell]$ injects into $E(\bF_{19})$ for $\ell \neq 19$. These facts force $E(\bQ)[2] = \bZ/2\bZ$, $E(\bQ)[3] = 0$, and $E(\bQ)[\ell] = 0$ for $\ell \neq 2, 3$. Hence $E_\tors(\bQ) = \langle Q \rangle$.

Now we consider the free part. The point $P$ was found using Cremona's mwrank library \cite{Cre-mwrank} included in {\tt Sage} \cite{sage-4.4.1}. We may check that $P \in E(\bQ)$ using a computer; this shows that $\rank\, E \geq 1$. (The point $P$ is reasonably difficult to find from scratch; indeed the standard functions for computing $E(\bQ)$ in both {\tt Magma} and {\tt Sage} fail to find $P$.)

To show that $\rank E \leq 1$ we use a standard $2$-descent strategy (see for example \cite[Ch.~III]{ST-ratpoints}). Consider the auxiliary curve
 \[ E' : y^2 = x^3 - 2ax^2 + (a^2 - 4b)x. \]
There are isogenies $\phi : E \to E'$ and $\hat\phi : E' \to E$ of degree $2$, and injections
\begin{gather*}
 E(\bQ)/\hat\phi(E'(\bQ)) \overset{\psi}{\injection} S \subset \bQ^*/(\bQ^*)^2, \\
 E'(\bQ)/\phi(E(\bQ)) \overset{\psi'}{\injection} S' \subset \bQ^*/(\bQ^*)^2,
\end{gather*}
where $S$ consists of the cosets $\delta(\bQ^*)^2$ for $\delta \divides 2 \dt 5 \dt 11 \dt 13 \dt 443$, and $S'$ of the cosets for $\delta \divides 11 \dt 113 \dt 127 \dt 443$ (these are the primes dividing $b$ and $a^2 - 4b$ respectively). We must determine which elements of $S$ and $S'$ arise from points in $E(\bQ)$ and $E'(\bQ)$. This is achieved by testing for the existence of rational points on the two families of quartic curves
\begin{gather}
\notag
 C_\delta : \delta w^2 = \delta^2 z^4 + \delta a z^2 + b, \qquad \delta \in S, \\
\label{eq:Cprime}
C'_\delta : \delta w^2 = \delta^2 z^4 - 2\delta a z^2 + (a^2 - 4b), \qquad \delta \in S'.
\end{gather}

We first consider the $C'_\delta$. If $443 \divides \delta$ then \eqref{eq:Cprime} has no solution in $\bQ_{443}$; if $(\delta/5) = -1$ then it has no solution in $\bQ_5$; and if $\delta \neq 1 \pmod 8$ then it has no solution in $\bQ_2$. These conditions rule out all but $\delta = 1$ and $\delta = -113\dt127$. These correspond to the classes in $E'(\bQ)/\phi(E(\bQ))$ of $O$ and the unique two-torsion point of $E'(\bQ)$; both have trivial image in $\hat\phi(E'(\bQ))/2E(\bQ)$.

Now we examine the $C_\delta$. For $\delta = 11\dt13$ there is the trivial rational point $z = 0$, $w = 2\dt5\dt11\dt443$, corresponding to the class of $Q$ in $E(\bQ)/\hat\phi(E'(\bQ))$. For $\delta = 2$ there is a (highly nontrivial) rational point corresponding to $P$, namely $z = (\frac 12x(P))^{1/2}$, $w = y(P) (2x(P))^{1/2}$. Rational points are automatic for $\delta = 1$ and $\delta = 2\dt11\dt13$ since the image of $\psi$ is a subgroup of $S$. We will show that $C_\delta(\bQ) = \emptyset$ for all other $\delta$.

Rewriting the equation for $C_\delta$ as $4\delta w^2 = (2\delta z^2 + a)^2 - (a^2 - 4b)$, we see immediately that $\delta > 0$ since $a^2 - 4b < 0$. Next, note that $(p/113) = 1$ for $p = 2, 11, 13, 443$, but $(5/113) = -1$. Thus if $5 \divides \delta$ we have $(\delta/113) = -1$; this is impossible as $v_{113}(a^2 - 4b) = 1$. Therefore $5 \ndivides \delta$.

To finish the argument for the $C_\delta$ it suffices to show that $C_\delta(\bQ) = \emptyset$ for $\delta = 11$, $443$ and $11\dt443$; the statement for the remaining $\delta$ will then follow automatically from the subgroup property.

Let $\delta = 11$, $443$, or $11\dt443$. Let $u = z^2$ and let $(u, w) = (u_0/t, w_0/t)$ be a rational point on the conic $\delta w^2 = \delta^2 u^2 + \delta a u + b$, where $u_0, w_0, t \in \bZ$. Intersecting the conic with a line of slope $X/Y$ through $(u_0/t, w_0/t)$, we obtain the parameterization $z^2 = f(X,Y)/g(X,Y)$ where
\begin{align*}
 f(X, Y) & = u_0 X^2 - 2w_0 XY + (ta + \delta u_0)Y^2, \\
 g(X, Y) & = t(X^2 - \delta Y^2),
\end{align*}
and where we may assume that $X, Y \in \bZ$ and $(X, Y) = 1$. Taking resultants, we find that any prime $p$ dividing $f(X, Y)$ and $g(X, Y)$ must divide $t$ or $a^2 - 4b = -11^2\dt113\dt127\dt443^2$. Thus
\begin{align}
\label{eq:eps-1}
 f(X, Y) & = \eps Z^2, \\
\label{eq:eps-2}
 g(X, Y) & = \eps W^2
\end{align}
for some $\eps \divides 11\dt113\dt127\dt443 t$, and some $W, Z \in \bZ$. We now consider each $\delta$ in turn, summarizing the local obstructions encountered for each possible $\eps$.

Let $\delta = 11$. We take $u_0 = 3 \dt 5^2 \dt 443$, $w_0 = 2^2 \dt 5 \dt 11 \dt 443$, $t = 1$. Then $\eps \divides 11\dt113\dt127\dt443$. If $443 \divides \eps$ then \eqref{eq:eps-2} has no solution in $\bQ_{443}$. If $11 \divides \eps$ then \eqref{eq:eps-1} has no solution in $\bQ_{11}$. If $(\eps/11) = -1$ then \eqref{eq:eps-2} has no solution in $\bQ_{11}$. This leaves $\eps \in \{1, 113, -127, -113\dt127\}$. For these $\eps$ we have $(\eps/443) = 1$. Equation \eqref{eq:eps-1} implies that $X = 14Y$ or $X = 110Y \pmod{443}$; both options contradict \eqref{eq:eps-2}.

Now consider $\delta = 443$. We take $u_0 = -3 \dt 11 \dt 13$, $w_0 = 11 \dt 13 \dt 443$, $t = 2$. Then $\eps \divides 2\dt11\dt113\dt127\dt443$. Suppose that $443 \ndivides \eps$. If $(\eps/443) = 1$ then \eqref{eq:eps-2} has no solution in $\bQ_{443}$, and if $(\eps/443) = -1$ then \eqref{eq:eps-1} has no solution in $\bQ_{443}$. Now let $\eps = 443\eps'$. If $(\eps'/443) = -1$ then \eqref{eq:eps-2} has no solution in $\bQ_{443}$. Now assume that $(\eps'/443) = 1$. Observe that $(p/443) = (p/11)$ for $p \in \{-1, 2, 113, 127\}$, but $(11/443) = -1$. This implies that either $11 \ndivides \eps'$ and $(\eps'/11) = 1$, or $11 \divides \eps'$ and $(\frac{\eps'}{11}/11) = -1$. In both cases, \eqref{eq:eps-1} forces $Y = 10X \pmod{11}$, and this contradicts \eqref{eq:eps-2}.

Finally let $\delta = 11\dt443$. We take $u_0 = 5^3 \dt 11$, $w_0 = 2 \dt 5 \dt 11 \dt 443$, $t = 3^2$. Then $\eps \divides 3\dt11\dt113\dt127\dt443$. If $11 \ndivides \eps$ then there are no solutions to \eqref{eq:eps-1} in $\bQ_{11}$. Suppose that $11 \divides \eps$. Then since $(p/13) = 1$ for $p \in \{-1, 3, 113, 127, 443\}$ and $(11/13) = -1$, we have $(\eps/13) = -1$; then \eqref{eq:eps-1} has no solution in $\bQ_{13}$.

This completes the $2$-descent. In particular, we have found that
 \[ |E(\bQ)/2E(\bQ)| = |E(\bQ)/\hat\phi(E'(\bQ))| \cdot |\hat\phi(E'(\bQ))/2E(\bQ)| = 4 \cdot 1 = 4, \]
and that $E(\bQ)/2E(\bQ)$ is generated by $P$ and $Q$. Moreover, for $R \neq O, Q$ the image of $x(R)$ in $\bQ^*/(\bQ^*)^2$ is one of $\{1, 2, 11\dt13, 2\dt11\dt13\}$.

At this stage we know that $\langle P, Q \rangle$ is of finite index in $E(\bQ)$; we must still check that it exhausts $E(\bQ)$. Suppose not; then for some prime $\ell$ and some $R \in E(\bQ)$ we have $\ell R = P$ or $\ell R = P + Q$. We cannot have $2R = P$ as $P$ is not divisible by $2$ in $E(\bF_3$); similarly $2R = P + Q$ is excluded by considering $E(\bF_7)$. Thus we may assume that $\ell$ is odd. If $\ell R = P + Q$ we replace $R$ by $R + Q$, so now may assume that $\ell R = P$ and $\ell (R+Q) = P + Q$.

In this case $e(R) \divides e(P) = 7^2 \dt 41 \dt 71 \dt 193$. From \eqref{eq:add-Q} we have
 \[ x(P + Q) = \frac{2 \dt 5^2 \dt 7^4 \dt 11 \dt 13 \dt 41^2 \dt 71^2 \dt 193^2}{3^2 \dt 83^2 \dt 6481^2}, \]
so similarly $e(R + Q) \divides 3 \dt 83 \dt 6481$. Moreover by \eqref{eq:add-Q} we have 
$$
\alpha(R) \alpha(R+Q) = b \dt e(R)^2 e(R+Q)^2.
$$ 
Since $(\alpha(R), e(R)) = (\alpha(R+Q), e(R+Q)) = 1$ this implies that $\alpha(R) = b_1 e(R+Q)^2$ and $\alpha(R+Q) = (b/b_1) e(R)^2$ for some $b_1 \divides b$. Since $P$ has singular reduction at $p = 2, 11, 443$, so does $R$, so $2\dt11\dt443 \divides b_1$. Similarly we find that $2\dt5\dt11\dt13 \divides (b/b_1)$. Comparing with the classes of $\bQ^*/(\bQ^*)^2$ found by the $2$-descent shows that we must have $b_1 = 2\dt11^2\dt443^2$ and $b/b_1 = 2\dt5^2\dt11\dt13$.

At this point we have reduced to $24$ possibilities for $e(R)$ and $8$ possibilities for $\alpha(R)$, and it is straightforward to check using a computer that the only pair defining a point on $E(\bQ)$ is $R = P$. Alternatively one may finish the argument using congruences. We sketch one quick way to do it: first prove that $3 \divides \alpha(R)$ by considering images in $E(\bF_3)$. Then for only 10 remaining values of $x(R)$ is $x^3 + ax^2 + bx$ a square in $\bQ_{11}$, and for only one of these is it a square in $\bQ_{31}$.
\end{proof}

\begin{prop}
\label{prop:integral}
The only solution to \eqref{eq:weierstrass} with $x, y \in \bZ[\frac12]$ is $x = y = 0$.
\end{prop}
\begin{proof}
Let $n \in \bZ$, $k \in \{0, 1\}$. We must prove that $x(nP + kQ) \notin \bZ[\frac12]$ for $n \neq 0$. We consider several cases.

First suppose that $k = 0$ and $n \neq 0$. Since $7 \divides e(P)$, also $7 \divides e(nP)$, so $x(nP) \notin \bZ[\frac12]$.

Next suppose that $k = 1$ and that $n$ is odd. Since $3 \divides e(P + Q)$, we have $3 \divides e(n(P + Q)) = e(nP + Q)$, so $x(nP + Q) \notin \bZ[\frac12]$.

Now suppose that $k = 1$ and $n = 2r$ where $r$ is odd. Since $79 \divides e(2P + Q)$, we have $79 \divides e(r(2P + Q)) = e(nP + Q)$, so $x(nP + Q) \notin \bZ[\frac12]$.

The last case is $k = 1$, $n = 0 \pmod 4$, $n \neq 0$. Write $n = 2^i r$ for some $i \geq 2$ and odd $r$. To continue the pattern we must find a prime $q$ playing the same role as $79$ from the previous case. For this, we first establish that
\begin{equation}
\label{eq:mod7}
 \alpha(2^j P) = \pm 4 \pmod 7 \qquad \text{for $j \geq 2$.}
\end{equation}
Indeed, one checks that $4P$ has nonsingular reduction for all $p$. The doubling formula \eqref{eq:double} and the comments regarding cancellation immediately following it then imply that $\alpha(2^{j+1}P) = \pm (\alpha(2^j P)^2 - b e(2^j P)^4)^2$ for all $j \geq 2$. Since $7 \divides e(2^j P)$ and $\alpha(4P) = \pm 4 \pmod 7$, identity \eqref{eq:mod7} follows by induction.

In particular $\alpha(2^i P) = \pm 4 \pmod 7$, so there must exist some prime $q$, not congruent to $1$ modulo $7$, dividing $\alpha(2^i P)$. We cannot have $q = 113$ or $q = 127$, as both of these are $1 \pmod 7$. Also, $q \notin \{2, 5, 11, 13, 443\}$, since for all of these primes the point $(0:0:1)$ is singular in $E(\bF_p)$, whereas $2^i P$ has nonsingular reduction for all $p$. Therefore $q$ is not a prime of bad reduction. From \eqref{eq:add-Q} we obtain $q \divides e(2^i P + Q)$. Finally, since $nP + Q = r(2^i P + Q)$, we have also $q \divides e(nP + Q)$, so that $x(nP + Q) \notin \bZ[\frac12]$.
\end{proof}
\begin{rema}
In several places in the above proof we use certain facts about $2P$ and $4P$. It is not necessary to compute their full coordinates, which are quite large (for example $\alpha(4P)$ has 256 digits). In every case it is possible to work $p$-adically to low precision. For example, to check that $4P$ has nonsingular reduction at $2$, it suffices to apply the doubling formula twice, using as input $a = 1 \pmod{2^3}$, $b = 28 \pmod{2^5}$ and $x(P) = 2 \pmod{2^4}$, to find that $x(2P) = 4 \pmod{2^5}$ and $x(4P) = 2^{-4} \pmod{2^{-3}}$.
\end{rema}

\bibliographystyle{plain}
\bibliography{hodgehilb}

\begin{thebibliography}{10}

\bibitem{beauville}
A.~Beauville.
\newblock Vari\'et\'es {K}\"ahleriennes dont la premi\`ere classe de {C}hern
  est nulle.
\newblock {\em J. Differential Geom.}, 18(4):755--782 (1984), 1983.

\bibitem{magma}
W.~Bosma, J.~Cannon, and C.~Playoust.
\newblock The {M}agma algebra system. {I}. {T}he user language.
\newblock {\em J. Symbolic Comput.}, 24(3-4):235--265, 1997.
\newblock Computational algebra and number theory (London, 1993).

\bibitem{Cre-mwrank}
J.~E. Cremona.
\newblock The mwrank library.
\newblock \\{\tt www.warwick.ac.uk/staff/J.E.Cremona/mwrank}.

\bibitem{EGL}
G.~Ellingsrud, L.~G{\"o}ttsche, and M.~Lehn.
\newblock On the cobordism class of the {H}ilbert scheme of a surface.
\newblock {\em J. Algebraic Geom.}, 10(1):81--100, 2001.

\bibitem{Fujiki}
A.~Fujiki.
\newblock On the de {R}ham cohomology group of a compact {K}\"ahler symplectic
  manifold.
\newblock In {\em Algebraic geometry, Sendai, 1985}, volume~10 of {\em Adv.
  Stud. Pure Math.}, pages 105--165. North-Holland, Amsterdam, 1987.

\bibitem{FulHar}
W.~Fulton and J.~Harris.
\newblock {\em Representation theory}, volume 129 of {\em Graduate Texts in
  Mathematics}.
\newblock Springer-Verlag, New York, 1991.
\newblock A first course, Readings in Mathematics.

\bibitem{Gott90}
L.~G{\"o}ttsche.
\newblock The {B}etti numbers of the {H}ilbert scheme of points on a smooth
  projective surface.
\newblock {\em Math. Ann.}, 286(1-3):193--207, 1990.

\bibitem{GHS}
V.~Gritsenko, K.~Hulek, and G.~K. Sankaran.
\newblock Moduli spaces of irreducible symplectic manifolds.
\newblock {\em Compos. Math.}, 146(2):404--434, 2010.

\bibitem{HTGAFA99}
B.~Hassett and Y.~Tschinkel.
\newblock Rational curves on holomorphic symplectic fourfolds.
\newblock {\em Geom. Funct. Anal.}, 11(6):1201--1228, 2001.

\bibitem{HT09}
B.~Hassett and Y.~Tschinkel.
\newblock Intersection numbers of extremal rays on holomorphic symplectic
  varieties, 2009.
\newblock Preprint, 24 pages, to appear in {\em Asian J. of Math.}

\bibitem{HTGAFA08}
B.~Hassett and Y.~Tschinkel.
\newblock Moving and ample cones of holomorphic symplectic fourfolds.
\newblock {\em Geometric and Functional Analysis}, 19(4):1065--1080, 2009.

\bibitem{HT10}
B.~Hassett and Y.~Tschinkel.
\newblock Hodge theory and {L}agrangian planes of generalized {K}ummer
  fourfolds, 2010.
\newblock Preprint, 28 pages.

\bibitem{LS}
M.~Lehn and Chr. Sorger.
\newblock The cup product of {H}ilbert schemes for {$K3$} surfaces.
\newblock {\em Invent. Math.}, 152(2):305--329, 2003.

\bibitem{Mark1}
E.~Markman.
\newblock On the monodromy of moduli spaces of sheaves on {$K3$} surfaces.
\newblock {\em J. Algebraic Geom.}, 17(1):29--99, 2008.

\bibitem{MarkCrelle}
Eyal Markman.
\newblock Generators of the cohomology ring of moduli spaces of sheaves on
  symplectic surfaces.
\newblock {\em J. Reine Angew. Math.}, 544:61--82, 2002.

\bibitem{Nakajima}
H.~Nakajima.
\newblock Heisenberg algebra and {H}ilbert schemes of points on projective
  surfaces.
\newblock {\em Ann. of Math. (2)}, 145(2):379--388, 1997.

\bibitem{Ran}
Z.~Ran.
\newblock Hodge theory and deformations of maps.
\newblock {\em Compositio Math.}, 97(3):309--328, 1995.

\bibitem{Sil-AEC}
J.~H. Silverman.
\newblock {\em The arithmetic of elliptic curves}, volume 106 of {\em Graduate
  Texts in Mathematics}.
\newblock Springer-Verlag, New York, 1992.
\newblock Corrected reprint of the 1986 original.

\bibitem{ST-ratpoints}
J.~H. Silverman and J.~Tate.
\newblock {\em Rational points on elliptic curves}.
\newblock Undergraduate Texts in Mathematics. Springer-Verlag, New York, 1992.

\bibitem{sage-4.4.1}
W.\thinspace{}A. Stein et~al.
\newblock {\em {S}age {M}athematics {S}oftware ({V}ersion 4.4.1)}.
\newblock The Sage Development Team, 2010.
\newblock {\tt http://www.sagemath.org}.

\bibitem{Verb}
M.~Verbitsky.
\newblock Cohomology of compact hyper-{K}\"ahler manifolds and its
  applications.
\newblock {\em Geom. Funct. Anal.}, 6(4):601--611, 1996.

\bibitem{Voisin}
Cl. Voisin.
\newblock Sur la stabilit\'e des sous-vari\'et\'es lagrangiennes des
  vari\'et\'es symplectiques holomorphes.
\newblock In {\em Complex projective geometry ({T}rieste, 1989/{B}ergen,
  1989)}, volume 179 of {\em London Math. Soc. Lecture Note Ser.}, pages
  294--303. Cambridge Univ. Press, Cambridge, 1992.

\bibitem{Wut-thesis}
C.~Wuthrich.
\newblock {\em The fine {S}elmer group and height pairings}.
\newblock PhD thesis, Cambridge, 2004.
\newblock \\{\tt
  http://www.maths.nottingham.ac.uk/personal/cw/download/phd.pdf}.

\end{thebibliography}

\end{document}